\documentclass[11pt]{amsart}
\usepackage{amsfonts}
\usepackage{graphicx,color}
\usepackage{amsthm}
\usepackage{amssymb,amsmath}

\title[Binary $(k,k)$-designs]{\bf Binary $(k,k)$-designs}

\date{\today}

\newtheorem{theorem}{Theorem}[section]

\newtheorem{corollary}[theorem]{Corollary}

\theoremstyle{definition}
\newtheorem{definition}[theorem]{Definition}
\newtheorem{example}[theorem]{Example}
\newtheorem{remark}[theorem]{Remark}

\author[T. Alexandrova]{Todorka Alexandrova$^{\dagger}$}
\address{Institute of Mathematics and Informatics, Bulgarian Academy of Sciences,
8 G Bonchev Str.,
1113  Sofia, Bulgaria}
\email{toty@math.bas.bg}
\author[P. Boyvalenkov]{Peter Boyvalenkov$^\dagger$}
\address{Institute of Mathematics and Informatics, Bulgarian Academy of Sciences,
8 G Bonchev Str., 1113  Sofia, Bulgaria \\
and Technical Faculty, South-Western University, Blagoevgrad, Bulgaria}
\email{peter@math.bas.bg}
\thanks{\noindent $^\dagger$ The research of the first two authors was supported, in part, by Bulgarian NSF under project KP-06-N32/2-2019. }
\author[A. Dimitrov]{Angel Dimitrov$^*$}
\address{Technical University Munich, Department of Mathematics, Boltzmannstrasse 3, Garching b. Munich, 85748, Germany}
\email{angel.dimitrov@tum.de}
\thanks{\noindent $^{*}$ The research of the third author was conducted during his internship in the Institute of Mathematics and Informatics of the Bulgarian Academy of Sciences. }

\date{}

\begin{document}
\maketitle

\begin{abstract}
We introduce and investigate binary $(k,k)$-designs -- combinatorial structures which are related to binary 
orthogonal arrays. We derive general linear programming bound and propose as a consequence a universal 
bound on the minimum possible cardinality of $(k,k)$-designs for fixed $k$ and length $n$. Designs which attain 
our bound are investigated.  
\end{abstract}



\section{Introduction}

Let $F=\{0,1\}$ be the alphabet of two symbols and $F_2^n$  the set of all binary
vectors $x=(x_1,x_2,\ldots,x_n)$ over $F$. The Hamming distance $d(x,y)$ between
points $x=(x_1,x_2,\ldots,x_n)$ and $y=(y_1,y_2,\ldots,y_n)$ from $F_2^n$ is equal to the number of
coordinates in which they differ. 

In considerations of $F_2^n$ as a polynomial metric space (cf. \cite{DL,Lev92,Lev-chapter}) it is convenient 
to use the ``inner product"
\begin{equation} \label{inner}
\langle x,y \rangle := 1-\frac{2d(x,y)}{n}
\end{equation}
instead of the distance $d(x,y)$. The geometry in $F_2^n$ is then related to the properties of the Krawtchouk polynomials
$\{Q_i^{(n)}(t)\}_{i=0}^n$ satisfying the following three-term recurrence relation 
\[ ntQ_i^{(n)}(t)=(n-i) Q_{i+1}^{(n)}(t)+iQ_{i-1}^{(n)}(t), \]
$i=1,2,\ldots,n-1$, with initial conditions $Q_0^{(n)}(t)=1$ and $Q_1^{(n)}(t)=t$. 

Any non-empty subset $C \subseteq F_2^n$ is called a code. Given a code $C \subset F_2^n$, the quantities 
 \begin{eqnarray}
\label{Mk0}
 M_i(C) &:=& \sum_{x,y\in C} Q^{(n)}_i(\langle x , y \rangle ) \\
&=& |C|+\sum_{x,y\in C, x \neq y} Q^{(n)}_i(\langle x , y \rangle ), \ i=1,2,\ldots,n, \nonumber
\end{eqnarray}
are called {\it moments} of $C$.

The well known positive definiteness of the Krawtchouk polynomials (see \cite{Del73,DL,Lev-chapter}) implies that $M_i(C) \geq 0$ for every $i=1,2,\ldots,n$.
The case of equality is quite important. 

\begin{definition} \label{T-designs} \cite{BBTZ17} Let $T \subset \{1,2,\ldots,n\}$. A code $C \subset F_2^n$ 
is called a $T$-design if 
\[ M_i=0 \mbox{ for all } i \in T. \]
\end{definition}

If $T=\{1,2,\ldots,m\}$ for some $m \leq n$, then $C$ is known as an $m$-design (see \cite{Del73,DL,Lev-chapter}), or 
(binary) orthogonal arrays of strength $m$ (cf. \cite{BBTZ17,Del73,DL,HSS,Lev-chapter}), 
or $m$-wise independent sets \cite{AGHP92}. The Euclidean analogs of the $(k,k)$-designs on $\mathbb{S}^{n-1}$ were considered earlier \cite{BOT17,Boy20,DS89,KP10,Wal17}. Orthogonal arrays have nice combinatorial properties which imply, in particular, a divisibility condition for $(k,k)$-designs
(Corollary \ref{div-k-k} below). 

The case of $T$ consisting of even integers was introduced and considered by Bannai et al. in \cite[Section 6.2]{BBTZ17} but (to the best of our knowledge)
the special case of the next definition is not claimed yet.

\begin{definition} \label{(k,k)-designs} If $C \subset F_2^n$ is a $T$-design with $T=\{2,4,\ldots,2k\}$, where $k \leq n/2$ is a positive integer,
then $C$ is called a $(k,k)$-design. In other words, $C$ is a $(k,k)$-design if and only if 
\[ M_i=0 \mbox{ for all } i=2,4,\ldots,2k. \]
\end{definition}

Thus, in this paper we focus on the special case when $T$ consists of several consecutive even integers begiining with 2. It is clear from 
the definition that any $(k,k)$-design is also an $(\ell,\ell)$-design for every $\ell =1,2,\ldots,k-1$. 

After recalling general linear programming techniques, we will derive and investigate an universal (in sense of 
Levenshtein \cite{Lev-chapter}) bound. 
More precisely, we obtain a lower bound on the quantity
\begin{equation} \label{min-card}
\mathcal{M}(n,k):=\min \{ |C|: C \subset F_2^n \mbox{ is a $(k,k)$-design}\}, 
\end{equation}
the minimum possible cardinality of a $(k,k)$-design in $F_2^n$, as follows:
\[ \mathcal{M}(n,k) \geq \sum_{i=0}^k {n-1 \choose i}. \]

The paper is organized as follows. In Section 2 we explain the relation between $(k,k)$-designs and antipodal $(2k+1)$-designs. Section 3 reviews 
the general linear programming bound and recalls the definition of so-called adjacent (to Krawtchouk) polynomials 
which will be important ingredients in our approach. Section 4 is devoted to our new bound. In Section 5 we discuss 
$(k,k)$-designs which attain this bound.  

\section{Relations to antipodal $(2k+1)$-designs}

Classical binary $m$-designs have nice combinatorial properties. 

\begin{definition} \label{def-m-des}
Let $C \subseteq F_2^n$ be a code and $M$ be a codeword 	matrix consisting of all vectors of $C$ as rows. Then 
$C$ is called an $m$-design, $1 \leq m \leq n$, if any set of $m$ columns of $M$ contains any $m$-tuple of $F_2^m$ the same number of times 
(namely, $\lambda:=|C|/2^m$). The largest positive integer $m$ such that $C$ is an $m$-design is called the
strength of $C$. The number $\lambda$ is called the index of $C$. 
\end{definition}


It follows from Definition \ref{def-m-des} that the cardinality of any $m$-design is divisible by $2^m$. This property 
implies a strong divisibility condition for a basic type of $(k,k)$-designs. 

\begin{definition} \label{antipod}
A code $C \subseteq F_2^n$ is called antipodal if for every $x \in C$ the unique point $y \in F_2^n$ 
such that $d(x,y)=n$ (equivalently, $ \langle x,y \rangle=-1$) also belongs to $C$. The point $y$ is denoted also by $-x$. 
\end{definition}

If $C \subset F_2^n$,
then the (multi)set of the points, which are antipodal to points of $C$ is denoted as usually by $-C$. A strong relation between antipodal $(2k+1)$-designs and $(k,k)$-designs is given as follows. 

\begin{theorem} \label{rel-anti-kk} Let $D \subset F_2^n$ be an antipodal $(2k+1)$-design. Let the 
code $C \subset F_2^n$ be formed by the following rule: from each pair $(x,-x)$ of 
antipodal points of $D$ exactly one of the points $x$ and $-x$ belongs to $C$. Then $C$ is
a $(k,k)$-design. Conversely, if $C \subset F_2^n$ is a $(k,k)$-design which does not possess a pair of 
antipodal points, then $D=C \cup -C$ is an antipodal $(2k+1)$-design in $F_2^n$. 
\end{theorem}

\begin{proof}
For the first statement we use in \eqref{Mk0} the antipodality of $D$, the relation $|C|=|D|/2$, and the 
fact that the polynomials $Q_{2i}^{(n)}(t)$ are even functions; i.e., $Q_{2i}^{(n)}(t)=Q_{2i}^{(n)}(-t)$ for every $t$, to see that
\[ M_{2i}(C)=\frac{M_{2i}(D)}{2}=0 \] 
for every $i=1,2,\ldots,k$. Therefore $C$ is a $(k,k)$-design (whichever is the way of choosing one of the 
points in pairs of antipodal points).

The second statement follows similarly.
\end{proof}

\begin{corollary} \label{div-k-k}
If $C \subset F_2^n$ is a  $(k,k)$-design which does not possess a pair of antipodal points, then $|C|$ is divisible by $2^{2k}$. 
\end{corollary}

\begin{proof}
By Definition \ref{def-m-des} it follows that $2^{2k+1}$ divides the cardinality of the antipodal $(2k+1)$-design $D$
constructed from $C$ as in Theorem \ref{rel-anti-kk}. Thus $|C|=|D|/2$ is divisible by $2^{2k}$. 
\end{proof}

\begin{example} \label{even-weight}
For even $n=2\ell$, the even weight code $D \subset F_2^n$ is an antipodal $(2\ell-1)$-design. Therefore, any code $C$
obtained as in Theorem \ref{rel-anti-kk} is an $(\ell-1,\ell-1)$-design. Obviously, $|C|=|D|/2=2^{n-2}=2^{2\ell-2}$.
We will be back to this example in Section 5. 
\end{example}

We note that Definition \ref{(k,k)-designs} shows that any $(2k)$-design is also a $(k,k)$-design. For small $k$, this relation gives some 
examples of $(k,k)$-designs with relatively small cardinalities (see Section 5). 

The $m$-designs in $F_2^n$ possess further nice combinatorial properties. For example, if a column of the matrix in Definition \ref{def-m-des} is deleted, the resulting matrix is still an $m$-design in $F_2^{n-1}$ with the same cardinality (possibly with repeating rows). 
Moreover, the rows with 0 in that column determine an $(m-1)$-design in $F_2^{n-1}$ of twice less cardinality. We do not know analogs 
of these properties for $(k,k)$-designs. 

\section{General linear programming bounds}

Linear programming methods were introduced in coding theory by Delsarte (see \cite{Del72,Del73}). The case of 
$T$-designs in $F_2^n$ was recently considered by Bannai et al \cite{BBTZ17}. 

The transformation \eqref{inner} means that all numbers $\langle x,y \rangle$ are rational and belong to the set 
\[ T_n:=\{-1+2i/n : i=0,1,\ldots,n\}. \]
We will be interested in values of polynomials in $T_n$. 

For any real polynomial $f(t)$ we consider its expansion in terms of Krawtchouk  polynomials
\[ f(t)=\sum_{j=0}^n f_j Q_j^{(n)}(t) \]
(if the degree of the polynomial $f(t)$ exceeds $n$, then $f(t)$ is taken modulo $\prod_{i=0}^{n}(t-t_i)$, where $t_i=-1+2i/n \in T_n$, $i=0,1,\ldots,n$). 
We define the following set of polynomials 
\[ F_{n,k}:=\{ f(t)\geq 0 \ \forall \, t \in T_n \, : \, f_0>0, f_j \leq 0, j=1,3,\ldots,2k-1 \mbox{ and } j \geq 2k+1\}. \]

The next theorem was proved (in slightly different setting) in \cite{BBTZ17}. We provide a proof here in order to make the paper self-contained. 

\begin{theorem} \label{lp-k,k} \cite[Proposition 6.8]{BBTZ17}
If $f \in F_{n,k}$, then  
\[ \mathcal{M}(n,k) \geq \frac{f(1)}{f_0}. \]
If a $(k,k)$-design $C \subset F_2^n$ attains this bound, then all inner products $\langle x,y \rangle$ of distinct $x,y \in C$
are among the zeros of $f(t)$ and $f_iM_i(C)=0$ for every positive integer $i$. 
\end{theorem}

\begin{proof}
Bounds of this kind follow easily from the identity
\begin{equation}
  \label{main}
  |C|f(1)+\sum_{x,y\in C, x \neq y} f(\langle x,y\rangle)
      = |C|^2f_0 + \sum_{i=1}^m f_i M_i(C) 
\end{equation}
(see, for example, \cite[Equation (1.20)]{Lev92}, \cite[Equation (26)]{Lev95}), which is true for every 
code $C \subset F_2^n$ and every polynomial $f(t)=\sum_{j=0}^m f_j Q_j^{(n)}(t)$. 

Let $C$ be a $(k,k)$-design and $f \in F_{n,k}$. We apply \eqref{main} for $C$ and $f$. 
Since  $M_{2j}(C)=0$ for $j=1,2,\ldots,k$, $M_i \geq 0$ for all $i$, and $f_j \leq 0$ for all odd $j$ and for all even
$j>2k$, the right hand side of \eqref{main} does not exceed $f_0|C|^2$. 
The sum in the left hand side is nonnegative because $f(t) \geq 0$ for every $t \in T_n$. Thus the left hand side
is at least $f(1)|C|$ and we conclude that $|C| \geq f(1)/f_0$. Since this inequality follows for every $C$, we have 
$\mathcal{M}(n,k) \geq f(1)/f_0.$

If the equality is attained by some $(k,k)$-design $C \subset F_2^n$ and a polynomial $f  \in F_{n,k}$, then 
\[ \sum_{x,y\in C, x \neq y} f(\langle x,y\rangle)=\sum_{i=1}^m f_i M_i(C)=0. \]
Since $f(t) \geq 0$ for every $t \in T_n$, we conclude that $f(\langle x,y\rangle)=0$
whenever $x,y \in C$ are distinct. Finally, $M_i(C) \geq 0$ for every $i$ and $f_i \leq 0$ for $i \not\in \{2,4,\ldots 2k\}$
yield $f_iM_i(C)=0$ for every positive integer $i$. 
\end{proof}

We will propose suitable polynomials $f(t) \in F_{n,k}$ in the next section. Key ingredients are certain polynomials 
$\{Q_i^{1,1}(t)\}_{i=0}^{n-2}$ ({\it adjacent} to the Krawtchouk ones) which were first introduced as such and investigated by 
Levenshtein (cf. \cite{Lev-chapter} and references therein). In what follows in this section we describe the derivation of these polynomials. 

The definition of the adjacent polynomials $\{Q_i^{1,1}(t)\}_{i=0}^{n-2}$ requires a few steps as follows (cf. \cite{Lev-chapter}). 
Let 
\[ T_i(u,v):=\sum_{j=0}^i {n \choose i} Q_i^{(n)}(u) Q_i^{(n)}(v) \]
be the Christoffel-Darboux kernel (cf. \cite{Sze}) for the Krawtchouk polynomials as defined in the Introduction. Then one defines
$(1,0)$-adjacent polynomials \cite[Eq. (5.65)]{Lev-chapter} by
\begin{equation} 
\label{adjacent-10} 
Q_i^{1,0}(t) := \frac{T_i(t,1)}{T_i(1,1)}, \ \ i=0,1,\ldots,n-1,
\end{equation}
(the Christoffel-Darboux kernel for the $(1,0)$-adjacent polynomials). For the final step, denote 
\begin{equation}
\label{kernelT10}
T_i^{1,0} (x,y) := \sum_{j=0}^i \frac{\left(\sum_{u=0}^j {n \choose u} \right)^2}{{n-1 \choose j} } Q_j^{1,0}(x)  Q_j^{1,0}(y)
\end{equation}
and define \cite[Eq. (5.68)]{Lev-chapter} 
\begin{equation} 
\label{adjacent-11} 
Q_i^{1,1}(t) := \frac{T_i^{1,0}(t,-1)}{T_i^{1,0}(1,-1)}, \ \ i=0,1,\ldots,n-2.
\end{equation}
The first few $(1,1)$-adjacent polynomials are
\[ Q_0^{1,1}(t)=1, \ \ Q_1^{1,1}(t)=t, \]
\[ Q_2^{1,1}(t)=\frac{n^2t^2-n+2}{n^2-n+2}, \ \ Q_3^{1,1}(t)=\frac{n^2t^3-(n-8)t}{n^2-n+8}. \]

Equivalently, the polynomials $\{Q_i^{1,1}(t)\}_{i=0}^{n-2}$ can be defined as the unique series of normalized (to have value 1 at 1)
polynomials orthogonal on $T_n$ with respect to the discrete measure
\begin{equation}
\label{KrawOrtho2} 
\frac{nq^{2-n}(1-t)(1+t)}{4(n-1)(q-1)}  \sum_{i=0}^n r_{n-i} \delta_{t_i},\end{equation}
where $\delta_{t_i} $ is the Dirac-delta measure at $t_i \in T_n$ \cite[Section 6.2]{Lev-chapter}).

Finally, we note the explicit formula (cf. \cite[Section 6.2]{Lev-chapter}, \cite[p. 281]{FL}) 
\begin{equation} 
\label{1-1-kraw}
Q_i^{1,1}(t) = \frac{K_i^{(n-2)}(z-1)}{\sum_{j=0}^i \binom{n-1}{j}}, 
\end{equation}
where $z=n(1-t)/2$, which relates the $(1,1)$-adjacent polynomials and the usual (binary) Krawtchouk polynomials 
\[ K_i^{(n)}(z):=\sum_{j=0}^i (-1)^j {z \choose j} {n-z \choose i-j}. \] 

It follows from \eqref{1-1-kraw} that the polynomials $Q_i^{1,1}(t)$ are odd/even functions for odd/even $i$
(this also follows from the fact that the measure \eqref{KrawOrtho2} is symmetric in $[-1,1]$). We will use this fact when
we deal with our proposal for a polynomial in Theorem \ref{lp-k,k}.

\section{A universal lower bound for $\mathcal{M}(n,k)$}

Using suitable polynomials in Theorem \ref{lp-k,k} we obtain the following universal 
bound. 

\begin{theorem} \label{univ-b}
We have
\[ \mathcal{M}(n,k) \geq \sum_{i=0}^k {n-1 \choose i}. \]
If a $(k,k)$-design $C \subset F_2^n$ attains this bound, then all inner products $\langle x,y \rangle$ of distinct $x,y \in C$
are among the zeros of $Q_k^{1,1}(t)$ and $|C|=\sum_{i=0}^k {n-1 \choose i}$ is divisible by $2^{2k}$. 
\end{theorem}

\begin{proof}
We use Theorem \ref{lp-k,k} with the polynomial $f(t)=\left(Q_k^{1,1}(t)\right)^2$ of degree $2k$ (so we have $f_i=0$ for $i \geq 2k+1$)
and arbitrary $(k,k)$-design in $F_2^n$. 
It is obvious that $f(t) \geq 0$ for every $t \in [-1,1]$. 
Since $Q_k^{1,1}(t)$ is an odd or even function, its square is an even function. Then $f_i=0$ for every odd $i$ 
and thus $f \in F_{n,k}$. The calculation of the ratio $f(1)/f_0$ gives the desired bound. 

If a $(k,k)$-design $C \subset F_2^n$ attains the bound, then equality in \eqref{main} follows (for $C$ and the above $f(t)$). 
Since $f_iM_i(C)=0$ for every $i$, the equality $|C|=f(1)/f_0$ is equivalent to 
\[ \sum_{x,y\in C, x \neq y} \left(Q_k^{1,1} (\langle x,y\rangle)\right)^2=0, \]
whence $Q_k^{1,1} (\langle x,y\rangle)=0$ whenever $x $ and $y$ are distinct points from $C$. 
The divisibility condition follows from Corollary \ref{div-k-k}. 
\end{proof}

\begin{remark}
Linear programming bounds (cf. (7)-(9) and Theorem 4.3 in \cite{Lev95}) with the polynomial 
$(t+1)\left(Q_k^{1,1}(t)\right)^2$ 
give the Rao \cite{Rao47} bound (see also \cite{HSS,Lev-chapter} and references therein) 
 for the minimum possible cardinality of $(2k+1)$-designs in $F_2^n$, that is $2\sum_{i=0}^k {n-1 \choose i}$.  
Thus our calculation of $f(1)/f_0$ quite resembles (and in fact follows from) 
the classical one \cite{Del73} (see also \cite[Section 2]{Lev-chapter}) 
by noting that, obviously, the value in one is two times less and the coefficient $f_0$ is the same because of the 
symmetric measure (equivalently, since $(t+1)\left(Q_k^{1,1}(t)\right)^2$ is equal to the sum of the odd function
$t\left(Q_k^{1,1}(t)\right)^2$ and our polynomial).  
\end{remark}


\section{On tight $(k,k)$-designs}

Following Bannai et al \cite{BBTZ17} we call {\it tight} every $(k,k)$-design in $F_2^n$ with cardinality $\sum_{i=0}^k {n-1 \choose i}$.
Example \ref{even-weight} provides tight $(\ell-1,\ell-1)$-designs for any even $n=2\ell$. Indeed, we have
\[ \sum_{i=0}^{\ell-1} {2\ell-1 \choose i}=\frac{1}{2} \sum_{i=0}^{2\ell} {2\ell-1 \choose i}=2^{2\ell-2}. \] 

Theorem \ref{rel-anti-kk} allows us to relate the existence of tight $(k,k)$-designs and tight $(2k+1)$-designs. 

\begin{theorem} \label{tight-2k+1}
For fixed $n$ and $k$, tight $(k,k)$-designs exist if and only if tight $(2k+1)$-designs exist.
\end{theorem}

\begin{proof}
If $C \subset F_2^n$ is a tight $(k,k)$-design, it can not possess a pair of antipodal points since $-1$ is not a zero of
$Q_k^{1,1}(t)$. Thus we may construct an antipodal $(2k-1)$-design $D \subset F_2^n$ with cardinality
\[ 2|C|=2\sum_{i=0}^k {n-1 \choose i}; \]
i.e., attaining Rao bound. 

Conversely, any tight $(2k+1)$-design in $F_2^n$ has cardinality $2\sum_{i=0}^k {n-1 \choose i}$ and is antipodal. 
By Theorem \ref{rel-anti-kk} it produces a tight $(k,k)$-design. 
\end{proof}

We proceed with consideration of the tight $(k,k)$-designs with $k \leq 3$.
The tight $(1,1)$-designs coexist with the Hadamard matrices due to a well known construction. 

\begin{theorem} \label{tight1}
Tight $(1,1)$-designs exist if and only if $n$ is divisible by 4 and there exists a Hadamard matrix of order $n$.
\end{theorem}

\begin{proof}
Let $C \subset F_2^n$ be a $(k,k)$-design with 
\[ 1+{n-1 \choose 1}=n \] \
points. Then $n$ is divisible by 4 and, moreover, 
since $Q_1^{1,1}(t)=t$, the only possible inner product is 0, meaning that the only possible distance is $n/2$. Therefore 
$C$ is a $(n,n,n/2)$ binary code. Changing $0 \to -1$ we obtain a Hadamard matrix of order $n$. Clearly, this works in the 
other direction as well.  
\end{proof}

Doubling a tight $(1,1)$-design gives a tight $3$-design which is clearly related to a Hadamard code $(n,2n,n/2)$. 
It is also worth to note that a Hadamard matrix of order $n+1$ defines a tight 2-design in $F_2^n$, which is an $(1,1)$-designs
with cardinality $n+1$ \cite[Theorem 7.5]{HSS}; i.e., exceeding our bound by 1. The divisibility condition now shows that 
this is the minimum possible cardinality for length $n \equiv 3 \pmod{4}$. 
Further examples of $(1,1)$-designs can be 
extracted from the examples in \cite{HKL06} and \cite{BDHSS20}, where linear programming bounds for codes with given minimum 
and maximum distances are considered. 

The classification of tight $(k,k)$-designs, $k \geq 2$, will be already as difficult combinatorial problem as the analogous problems for 
classical designs in Hamming spaces (see, for example \cite{BBTZ17,BOT17,GSV20,Nod79} and references therein). We present here the direct consequences of the linear programming approach combined with the divisibility condition of Corollary \ref{div-k-k}. 

\begin{theorem} \label{tight2}
Tight $(2,2)$-designs could possibly exist only for $n=m^2+2$, where $m \geq 3$ is a positive integer,
$m \equiv 2, 5, 6, 10, 11$ or $14 \pmod{16}$.
\end{theorem}

\begin{proof}
Let $C \subset F_2^n$ be a tight $(2,2)$-design. For $k=2$, we have 
\[ \mathcal{M}(n,2) \geq 1+{n-1 \choose 1}+{n-1 \choose 2}=(n^2-n+2)/2, \]
which means that $n^2-n+2$ is divisible by 32. 
This yields $n \equiv 6$ or $27 \pmod{32}$.

Looking at the zeros of $Q_2^{1,1}(t)$, we obtain $\pm \sqrt{n-2}/n  \in T_n$, whence
it follows that $n-2$ has to be a perfect square. Setting $n=m^2+2$, we obtain $m \equiv 2, 5, 6, 10, 11$ or $14 \pmod{16}$.
\end{proof}

The classification of tight 4-designs was recently completed by Gavrilyuk, Suda and Vidali \cite{GSV20} (see also \cite{Nod79}). The 
only tight 4-design is the unique even-weight code of length 5 (see Example \ref{even-weight}). It has cardinality 16, which is the minimum possibility 
for a $(2,2)$-design of length 5 since in this case our bound is 11 and the cardinality must be divisible by $2^4=16$.

\begin{theorem} \label{tight3}
Tight $(3,3)$-designs could possibly exist only for $n \equiv 8 \pmod{16}$ or $n \equiv 107 \pmod{128}$, where $n=(m^2+8)/3$, 
$m \geq 4$ is a positive integer, divisible by 4 and not divisible by 3, or $m \equiv 43 \pmod{64}$. 
The code obtained as in Theorem \ref{rel-anti-kk} from the binary Golay code $[24,12,8]$
is a tight $(3,3)$-design.
\end{theorem}

\begin{proof}
Let $C \subset F_2^n$ be a tight $(3,3)$-design. Then $2^6$ divides 
\[ |C|=1+{n-1 \choose 1}+{n-1 \choose 2}+{n-1 \choose 3}=\frac{n(n^2-3n+8)}{6}, \]
i.e. $n(n^2-3n+8)$ is divisible by $2^7$.
This gives $n \equiv 0 \pmod{8}$ or $n \equiv 107 \pmod{128}$.
Since $Q_3^{1,1}(t)$ has roots 0 and $\pm \sqrt{3n-8}/n$ (the later necessarily belonging to $T_n$; otherwise $C$ would be 
an equidistant code with the only allowed distance $n/2$), it follows that $3n-8$ is a 
perfect square. Setting $n=8u$ and $3n-8=m^2$, we easily see that $u$ has to be odd and $m$ cannot be multiple of 3. If
$n \equiv 107 \pmod{128}$, we obtain $m \equiv 43 \pmod{64}$. 

The necessary conditions are fulfilled for $n=24$, where the Golay code, which is a tight $7$-design, produces as in Theorem \ref{rel-anti-kk} 
a tight $(3,3)$-design of $2^{11}=2048$ points. 
\end{proof}

\section*{References}


\end{document}